\documentclass[11pt,english]{amsart}

\usepackage[T1]{fontenc}
\usepackage{amssymb,dsfont,fourier,baskervald}
\usepackage[text={6in,8.5in},centering]{geometry}

\usepackage[dvipsnames]{xcolor}   
\usepackage{xparse}
\usepackage{xr-hyper}
\definecolor{brightmaroon}{rgb}{0.76, 0.13, 0.28}
\usepackage[linktocpage=true,colorlinks=true,hyperindex,citecolor=BrickRed,linkcolor=blue]{hyperref}

\newtheorem{theorem}{Theorem}[section]

\newtheorem{lemma}[theorem]{Lemma}

\numberwithin{equation}{section}

\begin{document}
	
\author{Amartya Goswami}
\address{[1] Department of Mathematics and Applied Mathematics, University of Johannesburg, South Africa;
[2] National Institute for Theoretical and Computational Sciences (NITheCS), South Africa.}
\email{agoswami@uj.ac.za}
	
\title{Spectral hyperspaces of Krasner hyperrings}
	
\date{}
	
	
	
\maketitle
	
\begin{abstract}
The purpose of this note is to prove that the hyperspaces of proper hyperideals of Krasner hyperrings are spectral.
\end{abstract}
	
\section{Introduction and Preliminaries}
 
The initial exploration of multi-valued algebraic structures dates back to \cite{M34}, where the pioneering work introduced hypergroups as an extension of groups, accommodating multi-valued binary operations. Subsequently, in \cite{K47}, the notion of hyperrings was established. Since their inception, hyperrings have undergone extensive investigation across both algebraic and geometric domains. A comprehensive examination of various algebraic attributes of hyperrings, including their generalizations, is provided in \cite{DavLeo} (also see \cite{DS06}). To delve into the geometric applications of hyperrings, refer to \cite{CC10, CC11,CL-F03,J18}. Moreover, in \cite{KT22}, a study delves into the topological aspects of multiplication Krasner hypermodules.

To best of author's knowledge, spectral spaces over Krasner hyperrings have never been studied. This note aims  to provide an example of a spectral space in the context of hyperrings.  In order to do so, we have adapted an alternative technique (see Lemma \ref{cso}) that avoids the need to verify existence  of a quasi-compact open basis that is closed under finite intersections. To make the paper self-contained, we have supplied all the essential definitions and proofs.
 
Suppose that $R$ is a nonempty set and $\mathcal{P}^*(R)$ is the set of all nonempty subsets of $R$. 
A  \emph{commutative Krasner hyperring} is a system $(R,+,\cdot,-,0)$ such that

(I) $(R,+,0)$ is a \emph{canonical hypergroup}, that is, $+\colon R\times R\to \mathcal{P}^*(R)$ is a hyperoperation on $A$ satisfying the following properties for all $a,b,c \in R$:
\begin{enumerate}
	
\item $a+b=b+a;$
	
\item $a+(b+c)=(a+b)+c$;
	
\item  there exists $0\in R$ such that $a+0=\lbrace a\rbrace$;
	
\item  for every $a$, there exists a unique $-a\in R$ such that $0\in a-a;$
	
\item if $a\in b+c$, then $c\in -b+a$ and $a \in c-b$,
\end{enumerate}

(II) $(R, \cdot)$ is a commutative semigroup,

(III) $a \cdot 0 = 0,$ and

(IV) $a\cdot (b+c)=  a\cdot b + a \cdot c,$
for all $a,b,c \in R.$
A hyperring $R$ is called \emph{unital} if $R$ has a multiplicative identity, that is, there exists $1\in R$ such that $a\cdot 1=a$ for all $a\in R$. We write  $ab$ for $a\cdot b$.
Whenever we say Krasner hyperring, we  mean a commutative unital Krasner hyperring. 
A subhypergroup $I$ of a hyperring $R$ is called a \emph{hyperideal} of $R$ if $r a \in I$ for all $r\in R, a \in I.$  By $\mathcal{I}_R$, we denote the set of all hyperideals of $R$. 
The following result is well-known.

\begin{lemma}\label{isi}
Let $R$ be a hyperring. If $\{I_{\lambda}\}_{\lambda \in \Lambda}$ is a nonempty family of hyperideals of  $R,$ then  \[\sum_{\lambda \in \Lambda}I_{\lambda} = \left\{x \mid x \in \sum_{\lambda \in \Lambda}a_{\lambda}, a_{\lambda} \in I_{\lambda}\right\}\]
is also a hyperideal of $R$, where each term of the sum is taken over a finite subset of the index set $\Lambda$.
\end{lemma}

	

A hyperideal $I$ of $R$ is called \emph{proper} if $I\neq R$. By $\mathcal{I}^+_R$, we  denote the set of prepr hyperideals of $R$. A proper hyperideal $M$ of a hyperring $R$ is called \emph{maximal} if the only hyperideals of $R$ that contain $M$ are $M$ itself and $R.$ The result below is similar to that for rings.

\begin{lemma}\label{max}
Every proper  hyperideal $I$  of a hyperring $R$ is contained in a  maximal hyperideal of $R$.
\end{lemma}


Let us now recall a few relevant terminologies from topology. Let $L$ be a complete lattice and $x\in L$. A \emph{cover of} $x$ is a family $\{y_{\lambda}\}_{\lambda\in \Lambda}$ of elements of $L$ such that $x\leqslant \bigvee_{\lambda\in \Lambda}y_{\lambda}$. An element $x$ of $L$ is called \emph{compact} if every cover of $x$ has a finite subcover. A lattice is \emph{algebraic} if it is complete and every element is a least upper
bound of compact elements. It is easy to see that the complete lattice $\mathcal{I}_R$ is algebraic, for every hyperideal is a sum of finitely generated (and hence compact) hyperideals of $R$. This property of $\mathcal{I}_R$ is going to play a crucial role in our proof of the main result. 

A space is called \emph{quasi-compact} if every open cover of it has a finite subcover, or equivalently, the space satisfies the finite intersection property. In this definition of quasi-compactness, we do not assume the space is $T_2.$ A closed subset $S$ of a space $X$ is called \emph{irreducible} if $S$ is not the union of two properly smaller closed
subsets of $X$. A  space $X$ is called \emph{sober} if every non-empty irreducible closed subset $\mathcal K$ of $X$ is of the form: $\mathcal K=\mathcal{C}_{\{x\}}$, the closure of an unique singleton set $\{x\}$.

According to \cite{H69}, a \emph{spectral space} is a topological space that is  quasi-compact, sober,  admitting a basis of quasi-compact open subsets that is closed under finite intersections. We are going to use of the following topological lemma in the proof of our main result.

\begin{lemma}\label{cso}
A quasi-compact, sober, open subspace of a spectral space is spectral. 
\end{lemma}

\begin{proof} 
Suppose that $S$ is a quasi-compact, sober, open subspace of a spectral space $X$. Since $S$ is quasi-compact and sober, it is sufficient to prove that the set $\mathcal{O}_{S}$ of compact open subsets of $S$ forms a basis of a topology that is closed under finite intersections. It is obvious that a subset $T$ of $S$ is open in $S$ if and only if $T$ is open in $X$, and hence a subset $T$ of $S$ belongs to $\mathcal{O}_{S}$ if and only if $T$ belongs to $\mathcal{O}_{X}.$ Now 
using these facts, we argue as follows.
	
Let $U$ be an open subset of $S$. Since $U$ is also open in $X$, we have $U=\cup\, \mathcal{U},$ for some subset $\mathcal{U}$ of $\mathcal{O}_{X}.$ But each element of $\mathcal{U}$ being a subset of $U$ is a subset of $S$, and it belongs to $\mathcal{O}_{S}.$ Therefore, every open subset of $S$ can be presented as a union of compact open subsets of $S$. Now it remains to prove that $\mathcal{O}_{S}$ is closed under finite intersections, but this immediately follows from the fact that $\mathcal{O}_{X}$ is closed under finite intersections. 
\end{proof} 
 
Let $R$ be a Krasner hyperring. The  lower topology on  $\mathcal{I}^+_R$ is the topology for which the sets of the type: 
\[
\mathcal{V}_J=
\left\{I\in \mathcal{I}^+_R\mid J\subseteq I \right\}, \qquad (J\in \mathcal{I}_R);	
\]
form a subbasis of closed sets, 
and by $\mathcal{I}^+_R$, we  also denote the topological space, called a \emph{hyperspace}.

\begin{theorem}\label{mth}
Let $R$ be a Krasner hyperring. Then the hyperspace $\mathcal{I}^+_R$  is spectral.
\end{theorem}

\begin{proof} 
Taking $X=\mathcal{I}_R$ and $S=\mathcal{I}^+_R$ in Lemma \ref{cso}, it is sufficient to  prove the following:

\begin{enumerate}
	
\item \label{ngs}
$\mathcal{I}_R$ is a spectral hyperspace;
	
\item \label{oco} $\mathcal{I}^+_R$ is quasi-compact;
	
\item \label{tso} $\mathcal{I}^+_R$ is sober.
	
\item \label{top}  $\mathcal{I}^+_R$ is an open subhyperspace of the hyperspace $\mathcal{I}_R$.
\end{enumerate}
We impose a lower topology on $\mathcal{I}_R$ by considering the family $\left\{I\in \mathcal{I}_R\mid J\subseteq I \right\}_{J\in \mathcal{I}_R}$ as a closed subbasis.  To show (1), observe that  $\mathcal{I}_R$ is an algebraic lattice, and hence, the  spectrality of the hyperspace $\mathcal{I}_R$ follows from \cite[Theorem 4.2]{P94}.

Although the argument in the proof of (2) is  similar to the proof of quasi-compactness of $\mathrm{Spec}(R)$ (where $R$ is an unital commutative ring) endowed with Zariski topology, however, we need to consider the following two facts:
\begin{enumerate}
\item[$\bullet$] In order to show quasi-compactness, we must be assured of the fact that if   $\mathcal{V}_I=\emptyset$, then $I=R.$ This fact is equivalent to the fact that the set of maximal hyperideals of $R$ is contained in $\mathcal{I}^+_R$. Now, Lemma \ref{max} guarantees that this set of maximal hyperideals is nonempty.

\item[$\bullet$] Note that our argument must be with subbasic closed sets (because of the nature of the topology), and hence, we need to rely on Alexander subbasis theorem.
\end{enumerate}  
With these notes, let us now prove (2).
Suppose that  $\{K_{ \lambda}\}_{\lambda \in \Lambda}$ is a family of subbasic closed sets of a  hyperspace $\mathcal{I}^+_R$   such that $\bigcap_{\lambda\in \Lambda}K_{ \lambda}=\emptyset.$ Let $\{I_{ \lambda}\}_{\lambda \in \Lambda}\in\mathcal{I}_R$ such  that $\forall \lambda \in \Lambda,$  $K_{ \lambda}=\mathcal{V}_{I_{ \lambda}}.$  Since \[\bigcap_{\lambda \in \Lambda}\mathcal{V}_{I_{ \lambda}}=\mathcal{V}_{\sum_{\lambda \in \Lambda}I_{ \lambda}},\] we get  $\mathcal{V}_{\sum_{\lambda \in \Lambda}I_{ \lambda}}=\emptyset.$ The existence of the sum follows from Lemma \ref{isi}. This implies that the hyperideal $ \sum_{\lambda \in \Lambda}I_{ \lambda}$ must be equal to $R$. Then, in particular, we obtain $1=n_{ \lambda_1}+\cdots + n_{ \lambda_n},$ where $n_{ \lambda_i}\in I_{\lambda_i}$, for $i=1, \ldots, n$. This implies    $R=\sum_{  i \, =1}^{ n}I_{\lambda_i}.$ Hence,   $\bigcap_{ i\,=1}^{ n}K_{ \lambda_i}=\emptyset,$ and $\mathcal{I}^+_R$ is quasi-compact by Alexander subbasis theorem.  
	
To obtain 	
(3), first we prove the existence of generic points of irreducible closed subsets of the hyperspace $\mathcal{I}^+_R$, and for that it is sufficient to show that $\mathcal{V}_I=\mathcal{C}_I$, whenever $I\in\mathcal{V}_I$. Since $\mathcal{C}_I$ is the smallest closed set containing $I$, and since $\mathcal{V}_I$ is a closed set containing $I$, obviously then  $\mathcal{C}_I\subseteq \mathcal{V}_I$. 
For the reverse inclusion, if $\mathcal{C}_I= \mathcal{I}^+_R$, then 
\[ 
\mathcal{I}^+_R=\mathcal{C}_I\subseteq \mathcal{V}_I\subseteq \mathcal{I}^+_R.
\] 
This proves that $\mathcal{V}_I=\mathcal{C}_I$. Suppose that $\mathcal{C}_I\neq \mathcal{I}^+_R$. Since $\mathcal{C}_I$ is a closed set,  there exists an  index set, $\Lambda$, such that,  for each $\lambda\in\Lambda$, there is a positive integer $n_{\lambda}$ and hyperideals $I_{\lambda 1},\dots, I_{\lambda n_\lambda}$ of $R$ such that 
\[
\mathcal{C}_I={\bigcap_{\lambda\in\Lambda}}\left({\bigcup_{ i\,=1}^{ n_\lambda}}\mathcal{V}_{I_{\lambda i}}\right).
\]
Since  
$\mathcal{C}_I\neq \mathcal{I}^+_R,$ we  assume that ${\bigcup_{ i\,=1}^{ n_\lambda}}\mathcal{V}_{I_{\lambda i}}$ is non-empty for each $\lambda$. Therefore, $I\in   {\bigcup_{ i\,=1}^{ n_\lambda}}\mathcal{V}_{I_{\lambda i}}$ for each $\lambda$, and hence \[\mathcal{V}_I\subseteq {\bigcup_{ i=1}^{ n_\lambda}}\mathcal{V}_{I_{\lambda i}},\] that is, $\mathcal{V}_I\subseteq \mathcal{C}_I$ as desired. 
To obtain the uniqueness of the generic point, it is sufficient to prove that $\mathcal{I}^+_R$ is a $T_0$-space. Let $I$ and $I'$ be two distinct elements of $\mathcal{I}^+_R$. Then, without loss of generality, we may assume that $I\nsubseteq I'$. Therefore $\mathcal{V}_I$ is a closed set containing $I$ and missing $I'$. 

Finally, to show
(4), notice that by considering  lower topology on $\mathcal{I}_R$, we immediately obtain $\{R\}=\mathcal{V}_R=\mathcal{C}_R,$ and hence $\mathcal{I}_R \backslash\mathcal{I}^+_R$ is closed. Hence $\mathcal{I}^+_R$ is an open hypersubspace of the hyperspace $\mathcal{I}_R$. \hfill \end{proof}


\begin{thebibliography}{10} 
	
\bibitem{CC11}
A. Connes and C. Consani, The hyperring of ad\'{e}le classes, \textit{J. Number Theory} \textbf{131}(2) (2011), 159--194.

\bibitem{CC10}
P. Corsini, From Krasner hyperrings to hyperstructures: in search of an absolute arithmetic, in:
Casimir Force, \textit{Casimir Operators and the Riemann Hypothesis}, de Gruyter, (2010), 147--198.

\bibitem{CL-F03}
\bysame and V. Leoreanu-Fotea, \textit{Applications of hyperstructure theory}, vol. 5, Springer, 2003.

\bibitem{DS06}
\bysame and A. Salasi, \textit{A realization of hyperrings}, Comm. Algebra \textbf{34}(12) (2006) 4389--4400.

\bibitem{DavLeo} \bysame and  V. Leoreanu-Fotea,  \textit{Hyperring theory and applications}, International Academic Press, 2007.

\bibitem{H69} 
M. Hochster, Prime ideal structure in commutative rings, \textit{Trans. Am. Math. Soc.}, \textbf{142} (1969), 43--60.

\bibitem{J18}
J. Jun, Algebraic geometry over hyperrings, \textit{Adv. Math.}, \textbf{323} (2018), 142--192.

\bibitem{K47} M. Krasner, A class of hyperrings and hyperfields, \textit{Internat. J. Math. Math. Sci.}, \textbf{6}(2) (1983), 307--311.  

\bibitem{KT22}
\"{O}. Kulak and B. N. T\"{u}rkmen, Zariski topology over multiplication Krasner hypermodules, \textit{Ukrainian Math. J.}, \textbf{74}(4) (2022), 597--607.

\bibitem{M34}
F. Marty, Sur une g\'{e}n\'{e}ralization de la notion de groupe, in: \textit{8th Congress Math. Scandinaves}, Stockholm, (1934), 45--49.

\bibitem{P94} H. A. Priestley, Intrinsic spectral topologies, in: \textit{Papers on general topology and applications}
(Flushing, NY, 1992), \textbf{728},  78--95, New York Acad. Sci., New York, 1994.
\end{thebibliography}
\end{document}